 \documentclass[11pt]{amsart}
\usepackage{amsrefs}

\usepackage{epsfig}  		% For postscript
\usepackage{epic,eepic}       % For epic and eepic output from xfig
\usepackage{graphicx}

\usepackage{enumerate}
\usepackage{mathbbol}
\usepackage{amssymb}

\DeclareSymbolFontAlphabet{\mathbb}{AMSb}
\DeclareSymbolFontAlphabet{\mathbbl}{bbold}
\newenvironment{claimproof}[1][Proof of Claim]{\begin{proof}[#1]}{\end{proof}}

\newtheorem{lemma}{Lemma}[section]
\newtheorem*{lemma*}{Lemma}
\newtheorem{theorem}[lemma]{Theorem}
\newtheorem*{theorem*}{Theorem}
\newtheorem{corollary}[lemma]{Corollary}
\newtheorem{proposition}[lemma]{Proposition}

\newtheorem*{proposition*}{Proposition}

\newtheorem*{fact*}{Fact}

\newtheorem*{notation*}{Notation}
\newtheorem*{conventions*}{Conventions}
\newtheorem{remark}[lemma]{Remark}
\newtheorem*{remark*}{Remark}

\newtheorem*{corollary*}{Corollary}

\newtheorem*{conjecture*}{Conjecture}

\newtheorem*{problem*}{Problem}
\newtheorem{question}{Question}
\newtheorem*{question*}{Question}

\theoremstyle{definition}
\newtheorem{example}{Example}
\newtheorem*{example*}{Example}
\newtheorem{definition}[lemma]{Definition}
\newtheorem*{definition*}{Definition}

\theoremstyle{remark}
\newtheorem{claim}{Claim}
\newtheorem*{claim*}{Claim}

\newtheorem*{construction*}{Construction}

\newtheorem*{exercise*}{Exercise}

\numberwithin{equation}{section}

\newcommand{\N}{\mathbb{N}}

\newcommand{\Q}{\mathbb{Q}}

\newcommand{\set}[1]{\left\{ #1 \right\}}

\newcommand{\bs}{\backslash}

\newcommand{\meet}{\wedge}
\newcommand{\join}{\vee}

\renewcommand\AA{{\mathcal A}}

\newcommand\KK{{\mathcal K}}
\newcommand\LL{{\mathcal L}}

\newcommand\RR{{\mathcal R}}

\newcommand\UU{{\mathcal U}}

\newcommand\bbzero{\mathbbl 0} %mathbbol package
\newcommand\bbone{\mathbbl 1}  %mathbbol package

\newcommand\fraisse{Fra\"\i ss\'e }

\DeclareMathOperator{\Lift}{Lift}
\usepackage[small]{diagrams}

\begin{document}
	\title{Ramsey Expansions of $\Lambda$-Ultrametric Spaces}
	\author{Samuel Braunfeld}
	
	\begin{abstract}
	For a finite lattice $\Lambda$, $\Lambda$-ultrametric spaces are a convenient language for describing structures equipped with a family of equivalence relations.
	When $\Lambda$ is finite and distributive, there exists a generic $\Lambda$-ultrametric space, and we here identify a family of Ramsey expansions for that space. This then allows a description the universal minimal flow of its automorphism group, and also implies the Ramsey property for all of the homogeneous structures constructed in \cite{Braun}. A point of technical interest is that our proof involves classes with non-unary algebraic closure operations. As a byproduct of some of the concepts developed, we also arrive at a natural description of the known homogeneous structures in a language consisting of finitely many linear orders, thus completing one of the goals of \cite{Braun}.
	\end{abstract}
	

\maketitle

\section{Introduction}
   \newtheorem*{theorem:UMF}{\bf{Theorem \ref{theorem:UMF}}}
    \newtheorem*{corollary:RamseyPermutations}{\bf{Corollary \ref{corollary:RamseyPermutations}}}
     \newtheorem*{theorem:main}{\bf{Theorem \ref{theorem:mainthm}}}

 In \cite{Braun}, $\Lambda$-ultrametric spaces appeared as a convenient language for working with structures equipped with of a family of equivalence relations. There it was proven that, for $\Lambda$ finite and distributive, there exists a generic $\Lambda$-ultrametric space, in the sense of \fraisse theory. For example, when $\Lambda$ is a chain, this is the usual homogeneous ultrametric space, as in \cite{UM}. These spaces were used to produce homogeneous structures in a language consisting of linear orders, with lattice of $\emptyset$-definable equivalence relations isomorphic to $\Lambda$. The passage from $\Lambda$ to the language of linear orders ensured that every equivalence relation corresponding to a meet-irreducible in $\Lambda$ was made convex with respect to at least one of the linear orders added.
 
 This process of adding linear orders subject to convexity requirements looked suspiciously like what one would do in order to produce a Ramsey lift of the generic $\Lambda$-ultrametric space. This prompted the question whether the classes thus produced were Ramsey classes, and the related question of describing the universal minimal flow of the automorphism group of the generic $\Lambda$-ultrametric space.
 
 It soon became apparent that, rather than working with linear orders on a $\Lambda$-ultrametric space, there were technical advantages to working with certain partial orders on substructures of quotients, which we term \textit{subquotient orders}. In this paper, we define a lift by particular subquotient orders $\vec \AA_\Lambda^{min}$ of the class $\AA_\Lambda$ of all finite $\Lambda$-ultrametric spaces, and arrive at the following theorem.
 
 \begin{theorem:UMF}
 Let $\Lambda$ be a finite distributive lattice, $\AA_\Lambda$ the class of all finite $\Lambda$-ultrametric spaces, $\Gamma$  the \fraisse limit of $\AA_\Lambda$, and $\vec \Gamma^{min} = {(\Gamma, (<_{E_i})_{i=1}^n)}$ the \fraisse limit of $\vec \AA_\Lambda^{min}$. Then 
 \begin{enumerate}
 \item $\vec \AA_\Lambda^{\min}$ is a Ramsey class and has the expansion property relative to $\AA_\Lambda$.
 \item The logic action of $\text{Aut}(\Gamma)$ on $\overline{\text{Aut}(\Gamma) \cdot (<_{E_i})_{i=1}^n}$ is the universal minimal flow of $\text{Aut}(\Gamma)$.
 \end{enumerate}
 \end{theorem:UMF}
 
 The above theorem gives an explicit description of the universal minimal flow of $\text{Aut}(\Gamma)$ as isomorphic to the logic action on the full product of the following factors indexed by $i \in [n]$, where $n$ is the number of meet-irreducibles in $\Lambda$: the space of linear orders on $\Gamma$ if $<_{E_i}$ satisfies a certain condition (its top relation is $\bbone$), and otherwise the full infinite Cartesian power of that space. In particular, the universal minimal flow is metrizable.
 
 Most of the work in this theorem is directed toward establishing the Ramsey property. In the main theorem of this paper, we prove the Ramsey property for the broader class of all \textit{well-equipped lifts} of $\AA_\Lambda$, a condition used to capture the interdefinability of an expansion by generic subquotient orders with some expansion by linear orders.
 
  \begin{theorem:main} [Main Theorem]
  Let $\Lambda$ be a finite distributive lattice, $\AA_\Lambda$ the class of all finite $\Lambda$-ultrametric spaces, and $\vec \AA_\Lambda$ a well-equipped lift of $\AA_\Lambda$. Then $\vec \AA_\Lambda$ is a Ramsey class.
  \end{theorem:main}
  
  This theorem also provides a positive answer to the question of whether the amalgamation classes produced in \cite{Braun} are Ramsey classes.
  
  \begin{corollary:RamseyPermutations}
  The amalgamation classes corresponding to all the homogeneous finite-dimensional permutation structures constructed in \cite{Braun}*{Proposition 3.10} are Ramsey classes.
  \end{corollary:RamseyPermutations}
  
  Our main theorem is proven using tools from Hubi\v{c}ka and Ne{\v{s}}etril \cite{HN}. In particular, we use a combination and generalization of the encoding techniques used there to prove Ramsey theorems for the free product of Ramsey classes and for structures that have a chain of definable equivalence relations. 
   
   Perhaps the main point of technical interest is this: since we are dealing with an arbitrary finite distributive lattice $\Lambda$ of equivalence relations rather than just a chain of such, the algebraic closure operation in the structures we consider is non-unary, that is the algebraic closure of a set is not determined by the algebraic closures of the elements of the set. A non-unary algebraic closure significantly complicates applying the theorems of \cite{HN}, and consequently few classes with a non-unary algebraic closure have been proven to be Ramsey classes, although some examples were recently given in \cite{HN2}.
   
   The main point in the analysis of this closure is to show that, in an appropriate category, the closure of a finite set is finite. Rather than analyzing the closure operation explicitly, we derive this from the algebraic closure operation on imaginary elements in the generic $\Lambda$-ultrametric space (as in Lemma \ref{lemma:closure0}).
   
   Another point of considerable technical interest is our use of what we call a \textit{quantifier-free reinterpretation}, a generalization of the argument appearing in \cite{Bod}*{Section 4}, to transfer the Ramsey property between classes. The natural class our arguments would be carried out in has a linear order satisfying many constraints, and the reinterpretation technique allows us to argue in a class where the linear order is more nearly generic, thereby avoiding much bookkeeping.
  
   As a bonus, the development of the concepts of subquotient order and well-equipped lift enables us to complete the project initiated in \cite{Braun} (inspired by a question of Cameron \cite{Cameron}) of producing a census of "natural" homogeneous structures in a language consisting of finitely many linear orders, which we call \textit{finite-dimensional permutation structures}. 
   
   Although the main construction from \cite{Braun} was sufficient to produce, for every finite distributive lattice $\Lambda$, some homogeneous finite-dimensional permutation structure with lattice of $\emptyset$-definable equivalence relations isomorphic to $\Lambda$, there were already known finite-dimensional permutation structures that it could not produce. However, when the construction is modified to work with subquotient orders rather than linear orders, all known examples are captured, including those in the recently completed 3-dimensional classification \cite{3dim}. This naturally prompts the following question.
   
   \begin{question}
   Is every homogeneous finite-dimensional permutation structure with lattice of $\emptyset$-definable equivalence relations isomorphic to $\Lambda$ interdefinable with the Fra\"\i ss\'e limit of some well-equipped lift of the class of all finite $\Lambda$-ultrametric spaces?
   \end{question}
   
\textbf{Acknowledgements} I would like to thank Jan Hubi\v{c}ka informing me of the results of \cite{HN} as soon as it was prepared, as well as for subsequent discussion of the techniques used therein. I would also like to thank Gregory Cherlin for our many discussions on the material in this paper.
   
\section{Ramsey Classes and multi-amalgamation Classes}
Let $\KK$ be a class of structures closed under isomorphism. Given $A, B \in \KK$, let $\binom B A$ denote the set of substructures of $B$ that are isomorphic to $A$. We will say $\KK$ is a \textit{Ramsey class} if for any $n \in \N$ and $A, B \in \KK$, there is a $C \in \KK$ such that if $\binom C A$ is colored with $n$ colors, there is a $\widehat B \in \binom C B$ such that $\binom {\widehat B} A$ is monochromatic (we will often just say $\widehat B$ is monochromatic).

We now give an exposition of Theorem 2.2 in \cite{HN}, which provides sufficient conditions for proving a class is Ramsey. By \cite{Nes}, when a Ramsey class has the joint embedding property, it also has the amalgamation property, and this theorem of \cite{HN} is particularly useful for proving a Ramsey theorem for a class whose Fra\"\i ss\'e limit has a non-degenerate algebraic closure operation, in the sense that there exist sets which are not their own algebraic closure.

\begin{definition}
An $L$-structure $A$ is \textit{irreducible} if for every distinct $x, y \in A$, there is some $R \in L$ and a tuple $\vec t$ containing $x, y$ such that $R(\vec t)$ holds in $A$.

A homomorphism $f: A \to B$ is a \textit{homomorphism-embedding} if $f$ restricted to any irreducible substructure of $A$ is an embedding, i.e. $f$ is injective and for any $R \in L$, $R(x_1, ..., x_{r}) \Leftrightarrow R(f(x_1), ..., f(x_{r}))$, where $r$ is the arity of $R$.

Given an $L$-structure $C$ and a class $\KK$ of $L$-structures, we say $C'$ is a \textit{$\KK$-completion} of $C$ if $C' \in \KK$ and there is a homomorphism-embedding $f: C \to C'$.

Given an $L$-structure $C$, an irreducible $B \subset C$, and a class $\KK$ of $L$-structures, we say $C'$ is a \textit{$\KK$-completion of $C$ with respect to copies of $B$} if $C'$ is an irreducible $\KK$-structure and there is a function $f:C \to C'$ such that $f$ restricted to any $\hat B \in \binom C B$ is an embedding.
\end{definition}

\begin{definition}
Given a language $L$, a \textit{closure description} is a set of pairs $(R^U, B)$, where $R^U \in L$ is an $n$-ary relation, and $B$ is a non-empty irreducible $L$-structure on the set $\set{1, ..., m}$ for some $m \leq n$. We call $R^U$ a \textit{closure relation}, and the corresponding structure $B$ its \textit{root}.
\end{definition}

\begin{definition}
Given an $L$-structure $A$, an $n$-ary relation $R \in L$, and $k \leq n$, the \textit{$R$-out-degree} of a $k$-tuple $(x_1, ..., x_k) \in A^k$ is the number of tuples $(x_{k+1}, ..., x_n) \in A^{n-k}$ such that $R(x_1, ..., x_n)$ holds in $A$.  

Given a closure description $\UU$, we say that a structure $A$ is \textit{$\UU$-closed} if, for every $(R^U, B) \in \UU$, the $R^U$-out-degree of any tuple $\vec t$ of elements of $A$ is 1 if $\vec t$ is an embedding of $B$ into $A$, and 0 otherwise.
\end{definition}

Thus, in a $\UU$-closed structure, a closure relation can be thought of as a function assigning additional points to each copy of its root. The strong amalgamation condition in the following definition ensures that, in our applications, the closure relations are such that these functions generate the algebraic closure in the Fra\"\i ss\'e limit.

We are now ready for the main definition and theorem.

\begin{definition}
Let $\RR$ be a Ramsey class of finite irreducible $L$-structures, and let $\UU$ be a closure description in $L$. We say that a subclass $\KK \subset \RR$ is an \textit{$(\RR, \UU)$-multi-amalgamation class} if:
\begin{enumerate}
\item $\KK$ consists of finite $\UU$-closed $L$-structures.
\item $\KK$ is closed under taking $\UU$-closed substructures.
\item $\KK$ has strong amalgamation.
\item \textbf{Locally finite completion property}: Let $B \in \KK$ and $C_0 \in \RR$. Then there exists an $n=n(B, C_0)$ such that for any $\UU$-closed $L$-structure $C$ that satisfies the conditions below, there exists a structure $C'$ that is a $\KK$-completion of $C$ with respect to copies of $B$. The conditions required on $C$ are as follows.
\begin{enumerate}
\item $C_0$ is an $\RR$-completion of $C$.
\item Every substructure of $C$ with at most $n$ elements has a $\KK$-completion.
\end{enumerate}   
\end{enumerate} 
\end{definition}

\begin{theorem} [\cite{HN}*{Theorem 2.2}] \label{theorem:multiamalg}
Let $\RR$ be a Ramsey class. Then every $(\RR, \UU)$-multi-amalgamation class is a Ramsey class.
\end{theorem}

If we wish to prove a class $\KK$ of $L$-structures is Ramsey, the following theorem from \cite{NR} provides, in many cases, a suitable $\RR$ for applying Theorem \ref{theorem:multiamalg}.

\begin{theorem} \label{theorem:NR}
Let $L$ be a finite relational language, such that $<$ is a binary relation in $L$. The class of all finite $L$-structures in which $<$ is interpreted as a linear order is a Ramsey class.
\end{theorem}

\section{$\Lambda$-Ultrametric Spaces and Subquotient Orders}

The first part of this section recalls some relevant results from \cite{Braun}, and we then introduce the shift in language from linear orders to subquotient orders.

\begin{definition} Let $\Lambda$ be a complete lattice. A \textit{$\Lambda$-ultrametric space} is a metric space where the metric takes values in $\Lambda$ and the triangle inequality involves join rather than addition, i.e. $d(x, z) \leq d(x, y) \join d(y, z)$.
\end{definition}

\begin{theorem}\label{theorem:isomorphism}
For a given finite lattice $\Lambda$, there is an isomorphism between the category of $\Lambda$-ultrametric spaces and the category of structures consisting of a set equipped with a family of equivalence relations, closed under taking intersections in the lattice of all equivalence relations on the set, and labeled by the elements of $\Lambda$ in such a way that the map from $\Lambda$ to the lattice of equivalence relations is meet-preserving. Furthermore, the functors of this isomorphism preserve homogeneity.
\end{theorem}

Although we do not prove this theorem here, we will define the functors giving this isomorphism.

Given a system of equivalence relations as specified above, we get the corresponding $\Lambda$-ultrametric space by taking the same universe and defining $d(x, y) = \bigwedge \set{\lambda \in \Lambda | x E_\lambda y}$. In the reverse direction, given a $\Lambda$-ultrametric space, we get the corresponding structure of equivalence relations by taking the same universe and defining $E_\lambda = \set{(x,y) | d(x,y) \leq \lambda}$.

Because of this isomorphism, we will often conflate an element $\lambda \in \Lambda$ with the equivalence relation of being at distance $\leq \lambda$.

Since the lattices we are considering will always be finite, they will have a top and bottom element, denoted $\bbone$ and $\bbzero$, respectively. Thus, $d(x, y) = \bbzero$ iff $x = y$.

For every finite \textit{distributive} lattice $\Lambda$, a construction was given in \cite{Braun} producing a countable homogeneous finite-dimensional permutation structure $\Gamma$, such that the lattice of $\emptyset$-definable equivalence relations in $\Gamma$ is isomorphic to $\Lambda$. The structure $\Gamma$ is naturally presented as a $\Lambda$-ultrametric space, equipped with multiple orders. When $\Lambda$ is distributive, the class of all finite $\Lambda$-ultrametric spaces is an amalgamation class. The structure $\Gamma$ is constructed by taking the generic $\Lambda$-ultrametric space, and adding linear orders that are generic, except that they are required to be convex with respect to a prescribed set of equivalence relations corresponding to a chain of meet-irreducibles in $\Lambda$; enough such linear orders have to be added so that every meet-irreducible is convex with respect to at least one order, and there are further complications if $\bbzero$ (equality) is meet-reducible.

However, there are known countable homogeneous finite-dimensional permutation structures that cannot be produced by the construction outlined above. A slight modification to the language used in the construction remedies this, and the resulting construction produces all known countable homogeneous finite-dimensional permutation structures, as well as eliminating a special case the construction required when $\bbzero$ is meet-reducible. The idea behind the change of language is that when a linear order on $\Gamma$ is convex with respect to an equivalence relation $E$, it is better viewed as two partial orders: one that orders points within any given $E$-class, and one that encodes an order on $\Gamma/E$.

\begin{definition}
Let $X$ be a structure, and $E \leq F$ equivalence relations on $X$. A \textit{subquotient-order from $E$ to $F$} is a partial order on $X/E$ in which two $E$-classes are comparable iff they lie in the same $F$-class (note, this pulls back to a partial order on $X$). Thus, this partial order provides a linear order of $C/E$ for each $C \in X/F$. We call $E$ the \textit{bottom relation} and $F$ the \textit{top relation} of the subquotient-order. 
\end{definition}

Depending on the context, we will switch between considering a given subquotient order as a partial order on equivalence classes, or its pullback to a partial order on points. A special case of this is when the subquotient order has bottom relation equality, which amounts to equating $X$ with $X/=$.

Working at this level of generality requires a straightforward revision of the proof of amalgamation in \cite{Braun}*{Lemma 3.7}. The proof is actually simplified by the language change, yielding the following.

\begin{theorem} [\cite{3dim}] \label{theorem:amalg}
Let $\Lambda$ be a finite distributive lattice, and $\Gamma$ the generic $\Lambda$-ultrametric space. Then there is a homogeneous expansion of $\Gamma$ by finitely many subquotient orders, each of which has a meet-irreducible bottom relation, which is generic in a natural sense.
\end{theorem}

We now define two useful constructions with subquotient orders, and then give two examples of homogeneous finite-dimensional permutation structures not produced by the construction of \cite{Braun}, but which can be produced once linear orders are replaced by subquotient orders.

For the remainder of this section, if $x$ is an $E$-class, and $F$ an equivalence relation above $E$, then $x/F$ will represent the $F$-class containing $x$.

\begin{definition}
Let $<_{E, F}$ be a subquotient order with bottom relation $E$ and top relation $F$, and let  $<_{F, G}$ be a subquotient order with bottom relation $F$ and top relation $G$. Then the \textit{composition} of $<_{F, G}$ with $<_{E, F}$, denoted $<_{F, G}[<_{E, F}]$, is the subquotient order with bottom relation $E$ and top relation $F$ given by $x <_{F, G}[<_{E, F}] y$ iff either of the following holds.
\begin{enumerate}
\item $x$ and $y$ are in the same $F$-class, and $x <_{E, F} y$
\item $x$ and $y$ are in distinct $F$-classes, and $x/F <_{F, G} y/F$.
\end{enumerate}
\end{definition}

\begin{definition}
Let $<_{E, F}$ be a subquotient order with bottom relation $E$ and top relation $F$, and let $G$ be an equivalence relation lying between $E$ and $F$. Then the \textit{restriction of $<_{E, F}$ to $G$}, denoted $<_{E, F} \upharpoonright_G$, is the subquotient order with bottom relation $E$ and top relation $G$ given by $x <_{E, F} \upharpoonright_G y$ iff $x$ and $y$ are in the same $G$-class and $x <_{E, F} y$.
\end{definition}

\begin{example}
Let $\AA$ be the amalgamation class consisting of all finite structures in the language $\set{E, <_1, <_2}$, where $E$ is an equivalence relation, $<_1$ is a linear order, and $<_2$ is an $E$-convex linear order that agrees with $<_1$ on $E$-classes. 

 Let $\AA'$ be the class of all finite structures in the language $\set{E', <'_1, <'_2}$, where $E'$ is an equivalence relation, $<'_1$ is a subquotient order from $=$ to $\bbone$, and $<'_2$ a subquotient order from $E'$ to $\bbone$. This is also an amalgamation class, and its Fra\"\i ss\'e limit $\Gamma'$ is interdefinable with the Fra\"\i ss\'e limit $\Gamma$ of $\AA$. 

To define $\Gamma$ from $\Gamma'$, let $<_1 = <'_1$, and let $<_2 = <'_2[<'_1 \upharpoonright_E]$. To define $\Gamma'$ from $\Gamma$, let $<'_1 = <_1$, and let $x <'_2 y$ iff $\neg xEy$ and $x <_2 y$.
\end{example}

For the next example, we will use the following lemma, which also enters into the proof of Lemma \ref{lemma:sqoproduct}.

\begin{lemma} \label{lemma:movesqo}
Let $\Gamma$ be the generic $\Lambda$-ultrametric space. Let $E \in \Lambda$, with $E = F_1 \meet F_2$. Then a subquotient order $<_{F_1, F_1 \join F_2}$ on $\Gamma$ with bottom relation $F_1$ and top relation $F_1 \join F_2$ induces a definable subquotient order with bottom relation $E$ and top relation $F_2$.
\end{lemma}
\begin{proof}
We wish to define an order on $E$-classes within $F_2$-classes. Since $E = F_1 \meet F_2$, within a given $F_2$-class each $E$-class is in a distinct $F_1$-class, and they are all in the same $(F_1 \join F_2)$-class. Thus, we can define a subquotient order $<_{E, F_2}$ with bottom relation $E$ and top relation $F_2$ by $x <_{E, F_2} y \Leftrightarrow x/F_1 <_{F_1, F_1 \join F_2} y/F_1$.
\end{proof}

\begin{example}
For a more complex example of the use of subquotient orders, consider the full product $\Q^2$. This is a homogeneous structure with universe $\Q^2$ in the language $\set{E_1, E_2, <_1, <_2}$, where $E_1$ and $E_2$ are the relations defined by agreement in the first and second coordinates, respectively, $<_1$ is the standard lexicographic order on $\Q^2$, and $<_2$ is the standard antilexicographic order on $\Q^2$. The structure $(\Q, E_1, E_2, <_1, <_2)$ can also be presented in the language of 4 linear orders and no equivalence relations.

We could also express this in the same signature, but with $<_1$ a subquotient order with bottom relation $E_1$ and top relation $\bbone$, and $<_2$ a subquotient order with bottom relation $E_2$ and top relation $\bbone$. We can take the subquotient orders to be generic. In the resulting Fra\"\i ss\'e limit, we can define the standard lexicographic order by using Lemma \ref{lemma:movesqo} on $<_2$ to induce a subquotient order $<_{=, E_1}$ with bottom relation equality, and top relation $E_1$, and then taking the composition $<_1[<_{=, E_1}]$. We may similarly define the standard antilexicographic order.
\end{example}

These examples cannot be produced by the construction from \cite{Braun} when using linear orders rather than subquotient orders, since there the only constraints we put on the linear orders were convexity conditions, which involves forbidding substructures of order 3. However, in Example 1, we must forbid a substructure of order 2 to force $<_1$ and $<_2$ to agree between $E$-related points. In Example 2, we must forbid the following substructure of order 4 (as well as another symmetric substructure):
\begin{enumerate}
\item $x_1 E_1 x_2$, $y_1 E_1 y_2$, $\neg x_1 E_1 y_1$
\item $x_1 E_2 y_1$, $x_2 E_2 y_2$, $\neg x_1 E_2 x_2$ 
\item $x_1 <_1 x_2$, $y_2 <_1 y_1$
\end{enumerate}

\begin{definition}
Let $\Lambda$ be a finite distributive lattice, and let $L$ be a language consisting of relations for the distances in $\Lambda$ and finitely many subquotient orders, labeled with their top and bottom relations. We say that the language $L$ is \textit{$\Lambda$-well-equipped} if $E \in \Lambda$ appears as the bottom relation of some subquotient order in $L$ with distinct bottom and top relations iff $E$ is meet-irreducible, for every $E \in \Lambda$. 

If $\AA_\Lambda$ is the class of all finite $\Lambda$-ultrametric spaces, and $L$ a $\Lambda$-well-equipped language, we will call $\vec \AA_\Lambda$ a \textit{well-equipped lift} of $\AA_\Lambda$ if it consists of all finite $\Lambda$-ultrametric spaces equipped with subquotient orders from $L$.
\end{definition}

\begin{lemma} \label{lemma:sqoproduct}
Let $\Lambda$ be a finite distributive lattice, $\AA_\Lambda$ the class of all finite $\Lambda$-ultrametric spaces, and $\vec \AA_\Lambda$ a well-equipped lift of $\AA_\Lambda$, with Fra\"\i ss\'e limit $\vec \Gamma$. Then for every $E<F \in \Lambda$, $\vec \Gamma$ has a definable subquotient order with bottom-relation $E$ and top-relation $F$.
\end{lemma}

\begin{proof}
We prove this by downward induction in $\Lambda$. Take an arbitrary $E \in \Lambda$, and assume the claim is true for every element above $E$. 

We first note that it is sufficient, for every $F' \in \Lambda$ covering $E$, to construct a definable subquotient order $<_{E, F'}$ with bottom relation $E$ and top relation $F'$. Indeed, by the induction hypothesis, there is some definable subquotient order $<_{F', F}$ with bottom relation $F'$ and top relation $F$, and then the composition $<_{F', F}[<_{E, F'}]$ gives the desired definable subquotient order.  

First assume $E$ is meet-irreducible. Then there is a unique $F' \in \Lambda$ covering $E$. By assumption, there is some subquotient order $<_E$ with bottom relation $E$ and top relation some $F'' \geq F'$. Then the restriction $<_{E}\upharpoonright_{F'}$ is as desired.

Now assume that $E$ is meet-reducible, and let $F'$ cover $E$. Since $E$ is meet-reducible, there is some $F''> E$ such that $E = F' \meet F''$. By the induction hypothesis, there is a definable subquotient order $<_{F'', F' \join F''}$ with bottom relation $F''$, and top relation $F' \join F''$. Then Lemma \ref{lemma:movesqo} provides a definable subquotient order with bottom relation $E$ and top relation $F'$.
\end{proof}

\begin{corollary} \label{corollary:extendsqo}
Let $\Lambda$ be a finite distributive lattice, $\AA_\Lambda$ the class of all finite $\Lambda$-ultrametric spaces, and $\vec \AA_\Lambda$ a well-equipped lift of $\AA_\Lambda$, with Fra\"\i ss\'e limit $\vec \Gamma$. Given any subquotient order $<_{E}$ on $\vec \Gamma$ with bottom relation $E$ and top relation $F$, we can define on $\vec \Gamma$ a subquotient order $<'_{E}$ with bottom relation $E$ and top relation $\bbone$, in such a way that $x <_{E} y$ iff $x <'_{E} y$ and $x, y$ are in the same $F$-class. 
\end{corollary}
\begin{proof}
By Lemma \ref{lemma:sqoproduct}, there is a definable subquotient order $<_{F}$ with bottom relation $F$ and top relation $\bbone$. Then the composition ${<_{F}[<_{E}]}$ is as desired.
\end{proof}

\begin{remark} \label{rem:choice}
We will later find it useful to have made concrete choices when applying Lemma \ref{lemma:sqoproduct} and Corollary \ref{corollary:extendsqo}. In particular, given an enumeration ${(<_{G, i})}$ of the subquotient orders with bottom relation $G$ for every $G \in \Lambda$, we may always use subquotient orders that have 1 in the second index, with the possible exception of the specified subquotient order $<_E$ in Corollary \ref{corollary:extendsqo}. 
\end{remark}

\begin{proposition}
Let $\Lambda$ be a finite distributive lattice, $\AA_\Lambda$ be the class of all finite $\Lambda$-ultrametric spaces, and $\vec \AA_\Lambda$ a well-equipped lift of $\AA_\Lambda$, with Fra\"\i ss\'e limit $\vec \Gamma$. Then the relations of $\vec \Gamma$ are interdefinable with a set of linear orders.
\end{proposition}

\begin{proof}
For each $E \in \Lambda$, and each subquotient order $<_{E, i}$ in the language with bottom relation $E$, let $<'_{E, i}$ be a subquotient order as in Corollary \ref{corollary:extendsqo}. By Lemma \ref{lemma:sqoproduct}, let $<_{\bbzero, E}$ be a definable subquotient order with bottom relation equality and top relation $E$, and let $<''_{E, i}$ be the linear order given by the composition $<'_{E, i}[<_{\bbzero, E}]$.

Then, in the language consisting of the equivalence relations $E \in \Lambda$, the set of subquotient orders is interdefinable with the set of corresponding linear orders produced above. We may further use Lemma \ref{lemma:sqoproduct} to produce, for every $E \in \Lambda$, a definable linear order that is $E$-convex. Then, the equivalence relations can be interdefinably replaced with linear orders $\set{<^*_E}$ defined below.

For each $E \in \Lambda$:
\begin{enumerate}
\item Let $<_E$ be the definable linear order such that $E$ is $<_E$-convex
\item Let $<^*_E$ agree with $<_E$ within $E$-classes, and agree with the reverse of $<_E$ between $E$-classes.
\end{enumerate} 
\end{proof}
	
\section{The Classes $\KK_0$ and $\KK$}	
Let $\Lambda$ be a finite distributive lattice, ${\AA_\Lambda}$ the class of all finite $\Lambda$-ultrametric spaces, and $\vec \AA_\Lambda$ a well-equipped lift. In order to to prove the locally finite completion property required in Theorem \ref{theorem:multiamalg}, we will need to lift $\vec \AA_\Lambda$ to a linguistically more complex class. The first part of the lift, adding elements representing equivalence classes, is isolated below. It is similar to the lift employed in Lemma 4.28 of \cite{HN} for metric spaces with jumps, and is common in model theory.

A $\KK_0$-structure is meant to be viewed as follows: the elements of $P_{E, 1}$ represent the $E$-classes of a $\Lambda$-ultrametric space, and $U_{E, E'}(x, y)$ holds if $x$ represents an $E$-class and $y$ represents the $E'$-class containing $x$. The metric is not explicitly present in the language of the lift, since it is encoded by the family $\set{U_{E, E'}}$. 

\begin{definition}
Let $\LL_0 = \set{\set{P_{E, 1}}_{E \in \Lambda}, \set{U_{E,E'}}_{E<E' \in \Lambda}}$, be a relational language where the $P_{E, 1}$ are unary and the $U_{E, E'}$ are binary. Let $\UU_U$ be the following closure description: the $U_{E, E'}$ are closure relations, and the root of $U_{E, E'}$ is a single point $x$ such that $P_{E, 1}(x)$.

Let $\KK_0$ consist of all finite $\UU_U$-closed $\LL_0$-structures for which the following hold.
		
		\begin{enumerate}
			\item The family $\set{P_{E, 1}}_{E \in \Lambda}$ forms a partition
			\item If $U_{E,E'}(x,y)$, then $P_{E, 1}(x)$ and $P_{E', 1}(y)$
			\item (Coherence) If $E < E' < E'' \in \Lambda$ and $U_{E, E'}(x, y)$, then $U_{E', E''}(y, z)$ implies $U_{E, E''}(x, z)$
			\item (Downward semi-closure) If $P_{E, 1}(x)$ and $P_{E', 1}(x')$, then there is at most one $y$ such that $P_{E \meet E', 1}(y)$ and $U_{E \meet E', E}(y, x)$, $U_{E \meet E', E'}(y, x')$
		\end{enumerate}
			
\end{definition}

\begin{definition}
Let $\leq_U$ be the relation defined on a $\KK_0$-structure by $x \leq_U y$ if there are $E, E' \in \Lambda$ such that $U_{E, E'}(x, y)$. If $x \leq_U y$ and we wish to specify that $y$ is an $E'$-class, we will write $x/E' = y$.
\end{definition}
		
		\begin{proposition}
			Let $K \in \KK_0$ and let $x, x' \in K$.
			Suppose $x/E_1 = z_1 = x'/E_1$, $x/E_2 = z_2 = x'/E_2$. Then there exists $y \in K$ such that $x/(E_1\meet E_2) = y = x'/(E_1\meet E_2)$.
		\end{proposition} 
		\begin{proof}
			Since $K$ is $\UU_{U}$-closed, there are unique $y = x/(E_1 \meet E_2)$ and $y' = x'/(E_1 \meet E_2)$. By coherence, $y/E_1 = z_1=y'/E_1, y/E_2 = z_2=y'/E_2$. By downward semi-closure, $y=y'$.
		\end{proof}
	
		\begin{definition} For $x, y \in K_0$, define $\delta(x, y)$ to be the least $E$ such that $x/E = y/E$. By the above proposition, this is well-defined.
		\end{definition}
		
		\begin{proposition} \label{proposition:triangle}
			Let $K \in \KK_0$. Then $\delta$  satisfies the triangle inequality in $K$.
		\end{proposition}
		\begin{proof}
			Suppose $\delta(x_1, x_2) = F$, $\delta(x_2, x_3) = F'$. Let $a = x_1/(F \join F')$ and $b = x_3/(F \join F')$. Then $a = x_2/(F \join F') = b$, so $\delta(x_1, x_3) \leq F \join F'$.
		\end{proof}
		
	However, the function $\delta$ is technically not a $\Lambda$-ultrametric, or even a $\Lambda$-pseudoultrametric, since in general $\delta(x, x) \neq \bbzero$. Note that $\delta$ encodes all the information present in the family $\set{U_{E, E'}}$.

\begin{definition}
To any $\Lambda$-ultrametric space $A$, we associate a structure $A^{eq}$, such that if $A \in \AA_\Lambda$, then $A^{eq} \in \KK_0$, as follows.	
\begin{enumerate}
	\item The universe of $A^{eq}$ is $\sqcup_{E \in \Lambda} A/E$.
	\item For each $E \in \Lambda$, label the elements of $A/E$ with the predicate $P_{E, 1}$.
	\item For each $E, E' \in \Lambda$ with $E < E'$, let $U_{E,E'}(x,y)$ hold if $P_{E, 1}(x)$, $P_{E', 1}(y)$, and $y$ represents the $E'$-class containing the $E$-class that $x$ represents.
\end{enumerate}
\end{definition}
Note that this is only a fragment of the full model-theoretic $A^{eq}$, since we are not adding equivalence classes for equivalence relations definable on $A^n$ for $n>1$.

Conversely, to any $K \in \KK_0$, we can associate a structure $A_K \in \AA_\Lambda$. The following construction can be viewed as considering each point in $K$ as representing an equivalence class and picking a generic point in each class, i.e. points that are no closer to each other than necessary.

\begin{definition} \label{def:AK}
 Let $K \in \KK_0$. For each $x \in K$, create a corresponding point $x_A \in A_K$. Then, let $d(x_A, x_A) = \bbzero$, and let distances between distinct points be defined by $d(x_A, y_A) = \delta(x, y)$. By Proposition \ref{proposition:triangle}, the result is a $\Lambda$-ultrametric space.
\end{definition}
		
\begin{proposition} \label{proposition:eqsub}
Let $K \in \KK_0$. Then $K$ embeds into $(A_K)^{eq}$.	
\end{proposition}
\begin{proof}
	For each $x \in K$, if $P_{E, 1}(x)$, we map $x$ to $x_A/E \in (A_K)^{eq}$. 
	
	Suppose, for $x \in K$, that $P_{E, 1}(x)$, and let $y=x$. Then $\delta(x, y) = E = \delta(x_A/E, y_A/E)$. 
	
	Now suppose $x, y \in K$ with $P_{E, 1}(x)$ and $P_{E', 1}(y)$ and $x \neq y$. Let $\delta(x, y) = F$. Then $d(x_A, y_A) = F$. Then in $(A_K)^{eq}$, the least $G \in \Lambda$ such that $x_A/G = y_A/G$ is $G=F$. Since $E, E' < F$, this means $\delta(x_A/E, y_A/E') = F$ as well. Thus, our map preserves the family $\set{P_{E,1}}$ and $\delta$, and so gives an embedding of $K$ into $(A_K)^{eq}$.
\end{proof}

Thus $\KK_0$ is exactly the closure under $\UU_U$-closed substructure of the class obtained by applying the $eq$ operation to $\AA_\Lambda$. We call such structures \textit{upward-closed}. We now consider an additional closure condition.
		
\begin{definition} \label{def: dclosed}
We say $K \in \KK_0$ is \textit{downward closed} if for any $x, y \in K$ such that $P_{E, 1}(x)$, $P_{F, 1}(y)$, and $\delta(x, y) = E \join F$, there is some $z \in K$ such that $P_{E \meet F, 1}(z)$ and $z \leq_U x,y$. 
\end{definition}
		
	\begin{lemma} \label{lemma:closure0}
	Let $K \in \KK_0$. Then there is a finite $\KK_0$-structure $cl_0(K)$ such that
	\begin{enumerate}
	\item $K$ embeds into $cl_0(K)$
	\item $cl_0(K)$ is downward closed
	\end{enumerate}
	\end{lemma}
	\begin{proof}
	By Theorem \ref{theorem:amalg}, $\AA_\Lambda$ is an amalgamation class. Let $\Gamma$ be the Fra\"\i ss\'e limit of $\AA_\Lambda$. Embed $A_K$ into $\Gamma$. Then $(A_K)^{eq}$ is contained in $\Gamma^{eq}$. Let $cl_0(K)$ be the algebraic closure of $(A_K)^{eq}$ in $\Gamma^{eq}$.
	
	Given $x \in cl_0(K)$, with $P_{E, 1}(x)$, for any $E' > E$, $x/E'$ is definable from $x$ by the formula $\phi(y) = U_{E, E'}(x, y)$, and so is in its algebraic closure. Thus $cl_0(K)$ is $\UU_{U}$-closed.
	
	By Proposition \ref{proposition:eqsub}, $(1)$ already holds of $(A_K)^{eq}$.
	
	 We now prove $(2)$. Let $x', y' \in \Gamma$, with $d(x', y') = E \join F$. Since the structure $A = \set{x', y', z'}$, with $d(x', z') = E$, $d(y', z') = F$, $d(x', y') = E \join F$, satisfies the triangle inequality, we have $A \in \AA_\Lambda$. Thus, as $\Gamma$ is the Fra\"\i ss\'e limit of $\AA_\Lambda$, there is some $z' \in \Gamma$ such that $d(x', z') = E$ and $d(y', z') = F$. Given $x, y \in cl_0(K)$ as in Definition \ref{def: dclosed}, there exist $x', y' \in \Gamma^{eq}$ such that $P_{\bbzero, 1}(x')$, $P_{\bbzero, 1}(y')$,  $x'/E = x$, $y'/F = y$, and $\delta(x', y') = E \join F$. Then there is a $z' \in \Gamma^{eq}$ such that $\delta(x', z') = E$, $\delta(y', z') = F$. Thus $z' \leq_U x, y$, and so we may take $z=z'/(E \meet F) \leq_U x, y$. Finally, there can be at most one such $z$, since there is at most one $E \meet F$-class contained in any given $E$-class and $F$-class, and so $z$ is in the algebraic closure of $\set{x, y}$.
	\end{proof}

We now define the full class to which we will lift structures from $\vec \AA_\Lambda$. This will combine adding elements representing equivalence classes with the technique of duplicating the structure and connecting the parts by bijections used in Proposition 4.31 of \cite{HN} for structures with multiple linear orders. Since we are using subquotient orders instead of linear orders, we only need to duplicate part of the structure for each subquotient order.

The reason multiple structures are used to handle multiple linear orders is that Theorem \ref{theorem:NR}, which we plan to use to provide an $\RR$ for Theorem \ref{theorem:multiamalg}, provides a class with only a single linear order. Thus, in \cite{HN}, each linear order is placed on a single copy of the structure, and the copies are ordered one after another to form a single linear order.

The relation $D_{E_1, E_2}(x_1, x_2, y)$ in the definition below is meant to be viewed as stating that $x_1$ and $x_2$ represent an $E_1$ and $E_2$ class, respectively, and $y$ represents their intersection. This intersection of equivalence classes is the reason the algebraic closure operation in the class below will be binary.
	
	\begin{definition} \label{def:K}
	For each $E \in \Lambda$ let $N_E \geq 1$, and let $<_{1-types}$ be a linear order on $\set{(E, i) | E \in \Lambda, i \in [N_E]}$. Relative to these parameters, we define $\KK$, a class of structures in the relational language
		$$\LL = \LL_0 \cup \set{\set{P_{E, i}}_{E \in \Lambda, 2 \leq i \leq N_{E}}, \set{B_{E, i, j}}_{E \in \Lambda, i,j \in [N_E]}, \set{D_{E, E'}}_{E \neq E' \in \Lambda}, D^\exists, <}$$ 
	where the relations $P_{E, i}$ are unary, the $B_{E, i, j}$ are binary, the $D_{E, E'}$ are ternary, $D^\exists$ is binary, and $<$ is binary. Let $\KK$ consist of all finite $\LL$-structures for which the following hold.
			
	\begin{enumerate}
		\item The family $\set{P_{E, i}}_{E \in \Lambda, i \in [N_E]}$ forms a partition such that classes that agree in the first index have the same cardinality.
		\item The substructure on the points $x$ such that $P_{E, 1}(x)$ holds for some $E \in \Lambda$ is an $\LL$-expansion of a $\KK_0$-structure.
		\item $<$ is a linear order, which agrees with $<_{1-types}$ between 1-types, i.e. if $P_{E, i}(x), P_{F, j}(y)$, $(E, i) \neq (F, j)$, then $x < y \Rightarrow (E, i) <_{1-types} (F, j)$.
		\item $D^\exists(x, y)$ iff there exists a $z$ such that $U_{E, E \join E'}(x, z)$, $U_{E', E \join E'}(y, z)$.
		\item If $D_{E_1, E_2}(x_1, x_2, y)$, then $P_{E_1, 1}(x_1)$, $P_{ E_2, 1}(x_2)$, $D^\exists(x_1, x_2)$, $P_{E_1 \meet E_2, 1}(y)$, and $y \leq_U x_1, x_2$.
		\item $B_{E, i, j}$ is the graph of a bijection from the points of $P_{E, i}$ to the points of $P_{E, j}$.
		\item If $B_{E, i, j}(x, y)$ and $B_{E, j, k}(y, z)$, then $B_{E, i, k}(x, z)$.
	\end{enumerate}
	\end{definition}
	\begin{definition}		
	We also define a closure description $\UU_{\KK}$ for $\LL$, in which the relations $U_{E, E'}$, $B_{E, i, j}$, and $D_{E, E'}$ are closure relations. The root of $U_{E, E'}$ is a single point $x$ such that $P_{E, 1}(x)$. The root of $B_{E, i, j}$ is a single point $x$ such that $P_{E, i}(x)$. The root of $D_{E, E'}$ is a pair of points $x_1, x_2$ such that $P_{E, 1}(x_1)$, $P_{E', 1}(x_2)$, $D^\exists(x_1, x_2)$, $U_{E, E'}(x_1, x_2)$ if $E < E'$ or $U_{E', E}(x_2, x_1)$ if $E' < E$, and $x_1 < x_2$ if $(E, 1) <_{1-types} (E', 1)$ or $x_2 < x_1$ if $(E', 1) <_{1-types} (E, 1)$. 
	
	\end{definition}
	Although $\KK$ is not closed under taking substructures, the class of $\UU_K$-closed $\KK$-structures is closed under taking $\UU_K$-closed substructures.
	
	\begin{definition}The \textit{metric part} of $K \in \KK$ is the $\KK_0$-structure appearing in condition $(2)$ of Definition \ref{def:K}, with language $\LL_0$. 
	\end{definition}
	
	\begin{remark}
	For $K \in \KK$, we can assign a distance $\delta(x,y)$ between two points in the non-metric part of $K$ as well, by taking the distance between the points $x$ and $y$ are in bijection with in the metric part of $K$. 
	\end{remark}

	\begin{lemma} \label{lemma:Kextension}
	Let $K \in \KK$, let $K_0$ be the metric part of $K$, and let $K_0'$ be a $\KK_0$-structure containing $K_0$. Then there is a $\KK$-structure $K'$ such that $K \subset K'$ and the metric part of $K'$ is $K_0'$.
	
	Furthermore, if $K_0'$ is downward closed, $K'$ can be taken to be $\UU_{\KK}$-closed.
	\end{lemma}
\begin{proof}
For any $x, y \in K_0'$ such that for some $E, F \in \Lambda$, $P_{E, 1}(x)$, $P_{F, 1}(y)$, and $\delta(x, y) = E \join F$, add the relation $D^\exists(x, y)$. Furthermore, if $K_0'$ is downward-closed, there is a $z \leq_U x, y$ such that $P_{E \meet F}(z)$, so add the relation $D_{E, F}(x, y, z)$. 

Then, for every $x_1 \in K_0' \bs K_0$, perform the following. Let $E \in \Lambda$ be such that $P_{E, 1}(x_1)$.
	\begin{enumerate}
		\item for every $2 \leq i \leq N_E$, add a point $x_i$ to $P_{E, i}$
		\item for every $ i, j \in [N_E]$, add the relation $B_{E, i, j}(x_i, x_j)$  
	\end{enumerate}
Finally, complete $<$ arbitrarily to a linear order so that it still agrees with $<_{1-types}$ between 1-types.
\end{proof}
	
	\begin{lemma} \label{lemma:closure1}
	Let $K \in \KK$. Then there is a $\UU_{\KK}$-closed $\KK$-structure $cl(K)$ such that $K$ is a substructure of $cl(K)$.
	\end{lemma}
\begin{proof}
	Let $K_0$ be the metric part of $K$. Let $cl_0(K_0)$ be as in Lemma \ref{lemma:closure0}. Then apply Lemma \ref{lemma:Kextension} to $K$ with $K_0' = cl_0(K_0)$, and let $cl(K)$ be the resulting $K'$.
\end{proof}

\section{Transfer}

In this section, we show that to prove $\vec \AA_\Lambda$ is a Ramsey class, it is sufficient to prove that the class of $\UU_\KK$-closed $\KK$-structures is a Ramsey class.

The first definition describes how we lift an $\vec \AA_\Lambda$-structure to a $\KK$-structure.

\begin{definition} \label{def:prelift}		
Let $\vec A \in \vec{\AA_\Lambda}$, and let $\vec \Gamma$ be the Fra\"\i ss\'e limit of $\vec \AA_\Lambda$. Before we describe how to lift an $\vec \AA_\Lambda$-structure into $\KK$, we first fix the following parameters and notation.
\begin{enumerate}[(a)]
\item For each meet-irreducible $E \in \Lambda$, let $N_E$ be the number of subquotient orders with bottom-relation $E$, and for each meet-reducible $E \in \Lambda$, let $N_E = 1$.
\item Enumerate the subquotient orders with bottom-relation $E$ as $<_{E, i}$ for ${i \in [N_E]}$.
\item For each $E \in \Lambda$, choose $F'_E$ a cover of $E$, and for $E$ meet-reducible choose $F''_E > E$ such that $E = F'_E \meet F''_E$.
\item For each element of $\set{(E, i)|E \in \Lambda, i \in [N_E]}$ with $E$ meet-reducible, use the construction in Lemma \ref{lemma:sqoproduct}, with the above choices of $F'_E$ and $F''_E$, to produce  a quantifier-free formula $\phi_{E, i}$ defining a subquotient order with bottom relation $E$ and top relation $\bbone$ on $\vec \Gamma$.

For $E$ meet-irreducible, use Corollary \ref{corollary:extendsqo} instead. 

Finally, as noted in Remark \ref{rem:choice}, we may assume that whenever the construction has to choose between multiple subquotient orders with a given bottom relation, it chooses the first in our enumeration.
\item Fix a linear order $<_{1-types}$ on $\set{(E, i)| E \in \Lambda, i \in [N_E]}$. 
\item Let $A$ be the metric part of $\vec A$.
\end{enumerate}

Note that, although $\phi_{E, i}$ is defined on elements of $\vec A$, it naturally induces a linear order on the elements of $P_{E, i}$ in $A^{eq}$.
We now associate a $\KK$-structure to $\vec A \in \vec \AA_\Lambda$. Let $L_\KK(\vec A)$ be as follows:
		
	\begin{enumerate}
		\item Construct $A^{eq}$.
		\item For each $E \in \Lambda$, $2 \leq i \leq N_E$, create a copy of the elements of $P_{E, 1}$, and label the elements of that copy with the predicate $P_{E, i}$.
		\item For each $E \in \Lambda$, for each $i, j \in [N_E]$, let $B_{E, i, j}(x, y)$ if $P_{E, i}(x)$, $P_{E, j}(y)$ and $x$ and $y$ represent the same $E$-class.
		\item For each $E \in \Lambda$, $i \in [N_E]$, define $<$ on the elements of $P_{E, i}$ to agree with the order induced by $\phi_{E, i}$ on those elements.
		\item Extend $<$ to a total order by setting $x < y$ if $P_{E, i}(x)$, $P_{F, j}(y)$, and $(E, i) <_{1-types} (F, j)$.
	\end{enumerate}
\end{definition} 

This gives a canonical lifting from $\vec \AA_\Lambda$ to $\KK$. However, we would like to lift elements of $\vec \AA_\Lambda$ to \textit{$\UU_\KK$-closed} $\KK$-structures, which will be done as follows. First, note that we can also apply the $L_\KK$-construction to $\vec \Gamma$.

\begin{definition} \label{def:lift}
Let $\vec A \in \vec{\AA_\Lambda}$, and let $\vec \Gamma$ be the Fra\"\i ss\'e limit of $\vec \AA_\Lambda$. Embed $\vec A$ into $\vec \Gamma$. This induces an embedding of $L_\KK(\vec A)$ into $L_\KK(\vec \Gamma)$, and let $\Lift(\vec A)$ be the algebraic closure of $L_\KK(\vec A)$ in $L_\KK(\vec \Gamma)$.
\end{definition}

By the proof of Lemma \ref{lemma:closure0}, $\Lift(\vec A)$ will be $\UU_\KK$-closed.

\begin{proposition} \label{prop:liftfunc}
Suppose $\vec A, \vec B \in \vec \AA_\Lambda$ and $\vec A$ embeds into $\vec B$. This induces an embedding from $L_\KK(\vec A)$ into $L_\KK(\vec B)$, and there is an embedding of $\Lift(\vec A)$ into $\Lift(\vec B)$ that extends this embedding. In particular, $\Lift(\vec A)$ is well-defined up to isomorphism over $\vec A$.
\end{proposition}
\begin{proof}
The definition of the $\Lift$ operation is based on an embedding of $L_\KK(\vec B)$ into $L_\KK(\vec \Gamma)$. Given such an embedding, it induces a corresponding embedding of $L_\KK(\vec A)$, and relative to these embeddings, $\Lift(\vec A)$ will then be a substructure of $\Lift(\vec B)$.

For the final point, take the embedding of $L_\KK(\vec A)$ to be an isomorphism of $L_\KK(\vec A)$ with $L_\KK(\vec B)$, or its inverse.
\end{proof}

Note that the $\Lift$ operation produces only a subset of the structures in $\KK$, since the order cannot be generic within $1$-types, but is controlled by the $(\phi_{E, i})$ and the $(<_{E, i})$, which remain definable in the lifted structure as appropriate restrictions of $<$. The Ramsey property for $\vec \AA_\Lambda$ corresponds more directly to the Ramsey property for the class of lifted structures, but working with $\UU_\KK$-closed $\KK$-structures reduces much of the bookkeeping. The next definition will provide a way to transfer the Ramsey property from $\UU_\KK$-closed $\KK$-structures to the class of lifted structures (or rather its closure under $\UU_\KK$-closed substructure).

\begin{definition} \label{def:qfr}
Fix a relational language $L$. A \textit{quantifier-free reinterpretation scheme} is a family of quantifier-free $L$-formulas $\Phi = \set{\phi_R : R \in L}$ such that $\phi_R$ has $n_R$ free variables, where $n_R$ is the arity of R.

A quantifier-free reinterpretation scheme $\Phi$ naturally induces a function $f_\Phi:\set{\text{$L$-structures}} \to \set{\text{$L$-structures}}$, where $f_\Phi(A)$ is given by reinterpreting each $R \in L$ as $\phi_R$. We call such a function a \textit{quantifier-free reinterpretation}.

Given a class $\RR$ of $L$-structures and a quantifier-free reinterpretation $f_\Phi$, if $f_\Phi$ is a retraction when restricted to $\RR$, we call it a \textit{quantifier-free retraction on $\RR$}. The image of a quantifier-free retraction is a \textit{quantifier-free retract of $\RR$}
\end{definition}

\begin{example}
Let $L$ consist of two binary relations, $<_1, <_2$. Let $\RR$ be the class of all finite $L$-structures where $<_1, <_2$ are linear orders. Then the class $\KK_1 \subset \RR$ consisting of structures where $<_1 = <_2$ is a quantifier-free retract of $\RR$, induced by the quantifier-free reinterpretation scheme $\Phi = \{\phi_{<_1}(x_1, x_2) = x_1 <_1 x_2, \phi_{<_2}(x_1, x_2) = x_1 <_1 x_2 \}$. 

Similarly, the class $\KK_2 \subset \RR$ for which $<_2 = <_1^{opp}$ is a quantifier-free retract, induced by the quantifier-free reinterpretation scheme $\Phi = \{\phi_{<_1}(x_1, x_2) = x_1 <_1 x_2, \phi_{<_2}(x_1, x_2) = x_2 <_1 x_1 \}$.
\end{example}

The first of the above examples essentially appeared in \cite{Bod}, where it was used to argue that if a class of structures with two generic linear orders had the Ramsey property, one could forget one of those linear orders and keep the Ramsey property. We now generalize that argument.

\begin{lemma} \label{lemma:forget}
Let $\KK_1$ be a Ramsey class, and $\KK_2$ a quantifier-free retract of $\KK_1$ such that $\KK_2 \subset \KK_1$. Then $\KK_2$ is a Ramsey class.
\end{lemma}	
\begin{proof}
	Let $A, B \in \KK_2$. Then we also have $A, B \in \KK_1$, and so there is some $C_1 \in \KK_1$ witnessing the Ramsey property for $A, B \in \KK_1$. Let $C_2 \in \KK_2$ be the retract of $C_1$. We claim $C_2$ witnesses the Ramsey property for $A, B \in \KK_2$.
	
	Let $\chi_2$ be a coloring of $\binom{C_2}{A}$. Let $f_\Phi$ be the quantifier-free reinterpretation, as in Definition \ref{def:qfr}, and recall $f_\Phi$ restricts to the identity on copies of $A, B$. Define a coloring $\chi_1$ of $\binom{C_1}{A}$ by $\chi_1(A) = \chi_2(f_\Phi(A))$. Let $\widehat B \subset C_1$ be a monochromatic copy of $B$. Then $f_\Phi(\widehat B) \subset C_2$ is a monochromatic copy of $B$, since $f_\Phi$ is the identity on $\widehat B$. 
\end{proof}

The idea of the retraction we will use is that in any lifted structure, $<$ is determined by the $(\phi_{E, i})$, and certain restrictions of $<$. In a $\KK$-structure, we can take these restrictions of $<$, forget the rest of $<$, and then use the $(\phi_{E, i})$ to define a new order from the restrictions. This is carried out in detail in the next definition.
  
 In the following definition, given $x \in P_{E, i}$ and $F >E$, $x/F$ is the $y \in P_{F, 1}$ such that $y \geq_U x'$, where $B_{E, i, 1}(x, x')$. In order for formulas involving $x/F$ to be quantifier-free, we must consider the relations $B_{E, i, j}$ and $U_{E, E'}$ to be functions. These relations define functions in $\UU_\KK$-closed $\KK$-structures, which is the reason for restricting ourselves to $\UU_\KK$-closed structures below.

\begin{definition} \label{def:K'}
Let $F'_E$, $F''_E$ be as in Definition \ref{def:prelift}.

For each $E \in \Lambda$, we inductively define a formula $\psi_{E, i}$ which gives a linear order on $P_{E, i}$. The case $E = \bbone$ is trivial. Now assume we have defined such $\psi_{F, i}$ for all $F > E$. 

If $E$ is meet-irreducible, let $\psi_{E, i}(x, y) \Leftrightarrow (\delta(x, y) = F' \meet x < y) \join (\delta(x, y) > F'_E \meet \psi_{F'_E, 1}(x/F'_E, y/F'_E))$.

If $E$ is meet-reducible, let $\psi_{E, i}(x, y) \Leftrightarrow (x \neq y) \meet (((\delta(x, y) = F'_E \meet \psi_{F''_E, 1}(x/F''_E, y/F''_E)) \join (\delta(x, y) \neq F'_E \meet \psi_{F'_E, 1}(x/F'_E, y/F'_E)))$. 

Let $\Phi_< = \set{\phi_<}$, where $\phi_<$ is
$$
(\bigvee_{\substack{E \in \Lambda \\ i \in [N_E]}} P_{E, i}(x) \wedge P_{E, i}(y) \wedge \psi_{E, i}(x, y)) \vee (\bigvee_{\substack{E, F \in \Lambda \\ i \in [N_E], j \in [N_F] \\ (E, i) <_{1-types} (F, j)}} P_{E, i}(x) \wedge P_{F, j}(y))
$$

Note that $f_\Phi$ restricts to the identity on structures for which $<$ is suitably encoded by the $(\psi_{F, i})$ and certain restrictions of $<$ on $P_{E, i}$ for meet-irreducible $E$. We thus let $\KK'$ be the quantifier-free retract of $\UU_\KK$-closed $\KK$-structures under the above quantifier-free reinterpretation scheme.
\end{definition}

\begin{remark}
Any $\KK'$-structure is $\UU_\KK$-closed.
\end{remark}

\begin{proposition}
Let $\vec A \in \vec \AA_\Lambda$. Then $\Lift(\vec A) \in \KK'$.
\end{proposition}
\begin{proof}
It is clear that $\Lift(\vec A)$ is a $\UU_\KK$-closed $\KK$-structure. We must check that $f_{\Phi_<}$ is the identity on $\Lift(\vec A)$. But $\Phi_<$ was defined so as to make this the case. 
\end{proof}

\begin{proposition} \label{proposition:KtoK'}
Suppose the class of $\UU_{\KK}$-closed $\KK$-structures is a Ramsey class. Then $\KK'$ is a Ramsey class.
\end{proposition}
\begin{proof}
This follows by Lemma \ref{lemma:forget}.
\end{proof}

\begin{proposition} \label{proposition:representation}
Let $K \in \KK'$. Then there is an $\vec A_K \in \vec \AA_\Lambda$ such that $K$ embeds into $\Lift(\vec A_K)$.
\end{proposition}
\begin{proof}
	Let $K_0$ be the metric part of $K$. Taking $A_{K_0}$ as in Definition \ref{def:AK} gives a structure in $\AA_\Lambda$, which needs to be expanded by certain subquotient orders $<_{E, i}$ in order to obtain a structure in $\vec \AA_\Lambda$. 
	
	The $<_{E, i}$ are determined in the following manner: we know that $<_{E, i}$ should have bottom relation $E$, and let $E'$ be its prescribed top-relation. Recall that each $x \in K_0$ gives a point $x_A \in A_{K_0}$. Let $<^*_{E, i}$ be the partial order on points of $A_{K_0}/E$ of the form $x_A/E$ defined as follows. For $x, y \in K_0$,  let $x_A/E <^*_{E, i} y_A/E$ if there are points $x', y' \in P_{E, i}$ such that $B_{E, 1, i}(x, x')$ and $B_{E, 1, i}(y, y')$, we have $\delta(x, y) \leq E'$ and $x' < y'$. 
	
	Then let $<_{E, i}$ be an arbitrary extension of $<^*_{E, i}$ to a subquotient order of $A_{K_0}$ with bottom relation $E$ and top relation $E'$. Then the resulting structure is the desired $\vec A_K$. 
\end{proof}

Thus the class $\KK'$ is exactly the closure under $\UU_\KK$-closed substructure of the class obtained by applying the $\Lift$ operation to $\vec \AA_\Lambda$.
	
\begin{lemma} \label{lemma:K'toA}
Suppose $\KK'$ is a Ramsey class. Then $\vec{\AA_\Lambda}$ is a Ramsey class.
\end{lemma}
\begin{proof}
Let $\vec A, \vec B \in \vec\AA_\Lambda$. Then $\Lift(\vec A), \Lift(\vec B)$ are $\KK'$-structures, and so there is some $C \in \KK'$ witnessing the Ramsey property for $\Lift(\vec A), \Lift(\vec B)$. By possible enlarging $C$, we may assume it has the form $\Lift(\vec C)$ for some $\vec C \in \vec \AA_\Lambda$. We will show that $\vec C$ witnesses the Ramsey property for $\vec A, \vec B$.

Let $\chi$ be a coloring of $\binom{\vec C}{\vec A}$. We wish to lift $\chi$ to a coloring $\widehat \chi$ of $\binom{C}{\Lift(\vec A)}$. 

\begin{claim}
Let $\vec X \in \vec A_\Lambda$, $\widehat X \in \KK'$, and $\widehat X \cong \Lift(\vec X)$. Then there is a unique substructure $\vec X_1$ of $\widehat X$ such that $(\vec X_1, \widehat X) \cong (L_\KK(\vec X), \Lift(\vec X))$.
\end{claim}
\begin{claimproof}
Since $\widehat X$ is equipped with a linear order, it is rigid, and hence there is a unique isomorphism of $\Lift(\vec X)$ with $\widehat X$. The claim follows.
\end{claimproof}

For any $\widehat X \cong \Lift(\vec X)$, we define $ker(\widehat X)$ to be the unique substructure such that $(ker(\widehat X), \widehat X) \cong (L_\KK(\vec X), \Lift(\vec X))$. 

Also, given a structure $X$ of the form $L_\KK(\vec X)$ we define a map $L_X^{opp}$ from $P_{\bbzero, 1} \subset X$  to $\vec A_X$ as defined in Proposition \ref{proposition:representation}, which sends $x$ to the corresponding point $x_A$. Note that if $\vec X \subset \vec Y$, and thus $L_\KK(\vec X) \subset L_\KK(\vec Y)$, then $L_{L_\KK(\vec Y)}^{opp}[L_\KK(\vec X)] \cong \vec X$. Furthermore, if we identify $L_{L_\KK(\vec Y)}^{opp}[L_\KK(\vec Y)]$ with $\vec Y$, then $L_{L_\KK(\vec Y)}^{opp}[L_\KK(\vec X)] = \vec X$.

\begin{claim}
There is a coloring $\widehat \chi$ of $\binom{C}{\Lift(\vec A)}$ such that, for $\widehat A \in \binom{C}{\Lift(\vec A)}$
$$\widehat \chi(\widehat A) = \chi(L_{ker(C)}^{opp}[ker(\widehat A)])$$
\end{claim}
\begin{claimproof}
Because $\widehat A \subset C$, $ker(\widehat A) \subset ker(C)$. Then, since $\widehat A \cong \Lift(\vec A)$, we have $L_{ker(C)}^{opp}[ker(\widehat A)] \in \binom{\vec C}{\vec A}$.
\end{claimproof}

By the Ramsey property for $C$, there is $\widehat B \cong \Lift(\vec B)$ in $C$ which is $\widehat \chi$-monochromatic. We now check $L_{ker(C)}^{opp}[ker(\widehat B)]$ is $\chi$-monochromatic.

If $\vec A_1 \subset \vec B$ with $\vec A_1 \cong \vec A$, then $L_\KK(\vec A_1) \subset L_\KK(\vec B)$, and by Proposition \ref{prop:liftfunc} this extends canonically to an embedding of $\Lift(\vec A_1)$ into $\widehat B$. Thus $\chi(\vec A_1) = \widehat \chi(\widehat A_1)$, with $\widehat A_1$ the image of $\Lift(\vec A_1)$ in $\widehat B$. Thus $\vec B$ is $\chi$-monochromatic.
\end{proof}

Thus, we have the following.
\begin{corollary} \label{lemma:KtoA}
Suppose the class of $\UU_{\KK}$-closed $\KK$-structures is a Ramsey class. Then $\vec{\AA}_\Lambda$ is a Ramsey class.
\end{corollary}
	
\section{Ramsey Theorems}
We now use Theorem \ref{theorem:multiamalg} to prove the class of $\UU_\KK$-closed $\KK$-structures is a Ramsey class.

We first consider the downward-closed $\KK_0$-structures (Definition \ref{def: dclosed}).

\begin{lemma} \label{lemma:K0amalg}
The class of downward closed $\KK_0$-structures is a strong amalgamation class.
\end{lemma}
\begin{proof}
Let the downward closed $\KK_0$-structures $K_1, K_2$ be the factors of an amalgamation problem, and let $K^*$ be their free amalgam. 
\begin{claim*}
$K^* \in \KK_0$. 
\end{claim*}
\begin{claimproof}
Since the $U_{E, E'}$ are unary, and since both factors and the base are $\UU_U$-closed, $K^*$ is $\UU_U$-closed.

 We only check downward semi-closure, since the other constraints follow immediately from the fact that they are satisfied in each factor. Let $x, y \in K^*$ with $P_{E, 1}(x)$, $P_{F, 1}(y)$, and  $\delta(x, y) = E \join F$. 
 
 If $x, y$ are not in the same factor, then there is no $z \leq_U x, y$, since the base is $\UU_U$-closed. If they are in the same factor, then in each factor there is at most one $z \leq_U x, y$ such that $P_{E \meet F, 1}(z)$, since each factor is downward closed. Thus, the only possible contradiction would be if $x, y$ were in the base, and there were $z_1, z_2$ in the first and second factor, respectively, such that $z_i \leq_U x, y$ and $P_{E \meet F, 1}(z_i)$. But this is impossible, since the base is also downward closed.
 \end{claimproof}

Then $cl_0(K^*)$ as provided by Lemma \ref{lemma:closure0} is a downward closed strong amalgam.
\end{proof}

\begin{lemma} \label{lemma:Kamalg} 
The class of $\UU_{\KK}$-closed $\KK$-structures is a strong amalgamation class.
\end{lemma}
\begin{proof}
Let the $\UU_\KK$-closed $\KK$-structures $K_1, K_2$ be the factors of an amalgamation problem. Let $K^*$ be the free amalgam of $K_1, K_2$, and let $K^*_0$ be the metric part of $K^*$. 

\begin{claim*} We can complete $K^*$ to a $\KK$-structure, $K^+$.
\end{claim*}
\begin{claimproof}
 By the claim in Lemma \ref{lemma:K0amalg}, $K^*_0 \in \KK_0$. Because the base is $\UU_\KK$-closed, $\set{B_{E, i, j}}$ are the graphs of the desired bijections in $K^*$. For any $x, y \in K^*$ such that for some $E, F \in \Lambda$, $P_{E, 1}(x)$, $P_{F, 1}(y)$, and $\delta(x, y) = E \join F$, add the relation $D^\exists(x, y)$. Finally, complete $<$ to a linear order that agrees with $<_{1-types}$ between 1-types (this doesn't conflict with any transitivity constraints). The remaining constraints are satisfied since they are satisfied in each factor. Thus, the resulting structure, $K^+$, is in $\KK$.
 \end{claimproof}

Then $cl(K^+)$ as provided by Lemma \ref{lemma:closure1} is a strong amalgam.
\end{proof}

\begin{lemma} \label{lemma:lfc}
	Let $\RR$ be the class of all finite $\LL$-structures where $<$ is a linear order. Then the class of $\UU_\KK$-closed $\KK$-structures has the locally finite completion property with respect to $(\RR, \UU_\KK)$.
\end{lemma}
\begin{proof}
	 Let $B$ be a $\UU_\KK$-closed $\KK$-structure and $C_0 \in \RR$. Set $n(B, C_0) = 0$. Let $C$ be a $\UU_{\KK}$-closed $\LL$-structure with $\RR$-completion $C_0$. We first note that it is sufficient to produce $C'$ a $\KK$-completion of $C$ with respect to copies of $B$, since then $cl(C')$ as provided by Lemma \ref{lemma:closure1} will be a $\UU_{\KK}$-closed $\KK$-completion of $C$ with respect to copies of $B$.
	 
	 Since we only need to produce a $\KK$-completion of $C$ with respect to copies of $B$, rather than a $\KK$-completion, we may assume that $C$ is a union of copies of $B$ and that all relations are between points which lie in a common copy of $B$. However, assuming this means we may only assume that $C$ is \textit{$\UU_\KK$-semi-closed}, meaning that the $R$-out-degree, for any closure relation $R$, of a tuple $\vec t$ is at most 1 if $\vec t$ represents an embedding of the root of $R$ into $C$, and 0 otherwise. (We could actually assume full closure for the relations $U_{E, E'}$ and $B_{E, i, j}$, since they represent unary functions, but it will not be necessary.) We claim this places the following constraints on $C$.
	
	\begin{enumerate}
		\item The family $\set{P_{E, i}}$ forms a partition.
		\item $C$ is $\UU_U$-closed.
		\item The  $U_{E, E'}$ are coherent.
		\item $<$ is an irreflexive, asymmetric, acyclic relation. 
		\item $<$ agrees with $<_{1-types}$ between 1-types.
		\item If $U_{E, E'}(x, y)$, then $P_{E, 1}(x)$, $P_{E', 1}(y)$.
		\item If $D^{\exists}(x,y)$, then there are $E, E' \in \Lambda$ such that $P_{E, 1}(x)$, $P_{E', 1}(y)$, and $\delta(x, y) = E \join E'$.
		\item If $D_{E_1, E_2}(x_1, x_2, y)$, then we have $P_{E_1, 1}(x_1)$, $P_{ E_2, 1}(x_2)$, $D^\exists(x_1, x_2)$, $P_{E_1 \meet E_2, 1}(y)$, and $y \leq_U x_1, x_2$.
		\item If $x, y$ are such that $P_{E, 1}(x)$, $P_{E', 1}(y)$, and $D^\exists(x, y)$, then there is exactly one $z$ such that $D_{E, E'}(x, y, z)$.
		\item $B_{E, i, j}$ is the graph of a bijection from $P_{E, i}$ to $P_{E, j}$.
		\item If $B_{E, i, j}(x, y)$ and $B_{E, j, k}(y, z)$, then $B_{E, i, k}(x, z)$.
		\item $C$ is downward semi-closed.
	\end{enumerate}
	With the exceptions of $(3)$, $(4)$, $(10)$, $(11)$, and $(12)$, the constraints are immediate from the assumption that $C$ is $\UU_\KK$-semi-closed, is a union of copies of $B$, and all relations are between points that lie in the common copy of $B$.
	
	 Constraint $(4)$ follows from the assumption that there is a homomorphism-embedding from $C$ to $C_0$, and $<$ is a linear order on $C_0$.
	
	Before continuing, we observe that if $a \in C$ lies in a given copy of $B$ (perhaps one of several), and $U_{E, E'}(a, b)$ or $B_{E, i, j}(a, b)$, then $b$ lies in that same copy of $B$. For suppose $a \in \widehat B$, but $b \not\in \widehat B$. Then, since $\widehat B$ is $\UU_\KK$-closed, there is some $b' \in \widehat B$ such that we also have $U_{E, E'}(a, b')$ (resp. $B_{E, i, j}(a, b')$). But this is forbidden, since $C$ is $\UU_{\KK}$-semi-closed.
	
	We check constraint $(3)$. Suppose $U_{E, E'}(x, y)$, $U_{E', E''}(y, z)$. By our observation, $x, y, z$ all lie in a single copy of $B$, and so $U_{E, E''}(x, z)$. Constraint $(11)$ follows similarly.
	
	Constraint $(10)$ holds since each $B_{E, i, j}$ is a union of bijections, and $C$ is $\UU_\KK$-semi-closed.
	
	Finally, suppose constraint $(12)$ is violated, so there are $x, y \in C$ such that $P_{E, 1}(x)$, $P_{E', 1}(y)$, and $\delta(x, y) = E \join E'$, and there are distinct $z_1, z_2$ such that $P_{E \meet E', 1}(z_i)$ and $z_i \leq_U x, y$. By our earlier observation, we must have $x, y, z_1$ lying in a single copy of $B$, as well as $x, y, z_2$. Since $B$ is $\UU_\KK$-closed, we then have $D_{E, E'}(x, y, z_1)$ and $D_{E, E'}(x, y, z_2)$. But this is forbidden, since $C$ is $\UU_{\KK}$-semi-closed. 

	We define an equivalence relation $P$ on $C$ whose classes are the family $\set{P_{E, i}}$. Taking the transitive closure of $<$ gives a partial order, for which $P$ is a congruence and which agrees with $<_{1-types}$ between $P$-classes, and so $<$ can be completed to a linear order which is $P$-convex and agrees with $<_{1-types}$ between $P$-classes. Thus, after adding the relation $D^\exists$ where appropriate, we can complete $C$ to a $\KK$-structure $C'$.
\end{proof}

\begin{theorem} \label{theorem:KRamsey}
The class of $\UU_\KK$-closed $\KK$-structures is a Ramsey class.
\end{theorem}
\begin{proof}
By Lemmas \ref{lemma:Kamalg} and \ref{lemma:lfc}, the class of $\UU_\KK$-closed $\KK$-structures is an $(R, \UU_K)$-multi-amalgamation class, where $\RR$ is the class of finite $\LL$-structures where $<$ is interpreted as a linear order. Thus, by Theorems \ref{theorem:multiamalg} and \ref{theorem:NR}, it is a Ramsey class.
\end{proof}

\begin{theorem} \label{theorem:mainthm}
Let $\Lambda$ be a finite distributive lattice, $\AA_\Lambda$ the class of all finite $\Lambda$-ultrametric spaces, and $\vec \AA_\Lambda$ a well-equipped lift of $\AA_\Lambda$. Then $\vec \AA_\Lambda$ is a Ramsey class.
\end{theorem}
\begin{proof}
By Theorem \ref{theorem:KRamsey} and Corollary \ref{lemma:KtoA}, we are done.
\end{proof}

\begin{corollary} \label{corollary:RamseyPermutations}
The amalgamation classes corresponding to all the homogeneous finite-dimensional permutation structures constructed in \cite{Braun}*{Proposition 3.10} are Ramsey classes.

In particular, for every finite distributive lattice $\Lambda$, there is a Ramsey class such that $\Lambda$ is isomorphic to the lattice of $\emptyset$-definable equivalence relations in the class's \fraisse limit.
\end{corollary}
\begin{proof}
For the first part, all such classes are representable as well-equipped lifts of the class of all $\Lambda$-ultrametric spaces for some finite distributive $\Lambda$.

For the second part, by \cite{Braun}*{Theorem 3.1}, for every such $\Lambda$, there is a homogeneous finite-dimensional permutation structure with lattice of equivalence relations isomorphic to $\Lambda$.
\end{proof}

\section{The Expansion Property}
We now identify a Ramsey lift of $\AA_\Lambda$ with the expansion property as defined below, and use this to compute the universal minimal flow of $Aut(\Gamma)$, where $\Gamma$ is the \fraisse limit of $\AA_\Lambda$. The following definitions and theorem are from \cite{NVT}, extending the work of \cite{KPT}.

\begin{definition}
Given an $L$-structure $F$, we let $\text{Age}(F)$ denote the set of all $L$-structures isomorphic to a finite substructure of $F$.
\end{definition}

\begin{definition}
Let $L$ be a language, and $L^*$ a countable relational expansion of $L$. Let $F$ be a homogeneous $L$-structure. Then an $L^*$-expansion $F^*$ of $F$ is \textit{precompact} if any $A \in \text{Age}(F)$ has only finitely many $L^*$-expansions in $\text{Age}(F^*)$.
\end{definition}

\begin{definition}
Let $F$ be a homogeneous structure, and $F^*$ a precompact relational expansion of $F$. Then $\text{Age}(F^*)$ has the \textit{expansion property} relative to $\text{Age}(F)$ if for every $A \in \text{Age}(F)$ there exists a $B \in \text{Age}(F)$ such that, for any $L^*$-expansions $A^*$ of $A$ and $B^*$ of $B$ in $\text{Age}(F^*)$, $A^*$ embeds into $B^*$.
\end{definition}

\begin{definition}
Let $L$ be a language, and $L^* = L \cup \set{R_i}_{i \in I}$ be a relational expansion. Let $a(i)$ be the arity of $R_i$. Given a homogeneous $L$-structure $F$, we define $P^*$ as
$$P^* = \prod_{i \in I}\set{0, 1}^{F^{a(i)}}$$

We may define a group action of $\text{Aut}(F)$ on a given factor as follows: for $g \in \text{Aut}(F)$ and $S_i \in F^{a(i)}$, $g \cdot S_i(y_1, ..., y_{(a_i)}) \Leftrightarrow S_i(g^{-1}(y), ..., g^{-1}(y_{a(i)}))$.

Finally, we may define the \textit{logic action} of $\text{Aut}(F)$ on $P^*$ as given by applying the above action componentwise.
\end{definition}

\begin{theorem} [\cite{NVT}*{Theorems 4, 5}] \label{theorem:NVT}
Let $L$ be a language, $L^* = L \cup \set{R_i}_{i \in I}$ be a countable relational expansion, and $F$ a homogeneous $L$-structure. Let $F^* = (F, \vec R^*)$ be a precompact $L^*$-expansion of $F$. Then we have the following equivalence, with the closure taken in $P^*$.
\begin{enumerate}
\item The logic action of $Aut(F)$ on $\overline{Aut(F) \cdot \vec R^*}$ is minimal.
\item $\text{Age}(F^*)$ has the expansion property relative to $\text{Age}(F)$.
\end{enumerate}

Furthermore, the following are also equivalent.

\begin{enumerate}
\item The logic action of $Aut(F)$ on $\overline{Aut(F) \cdot \vec R^*}$ is its universal minimal flow.
\item $\text{Age}(F^*)$ has the Ramsey property and the expansion property relative to $\text{Age}(F)$.
\end{enumerate}
\end{theorem}

\begin{proposition}
Suppose $F^*$ is a precompact expansion of a homogeneous structure $F$, and for every $A^* \in \text{Age}(F^*)$, there is a $B \in \text{Age}(F)$ such that for any expansion $B^* \in \text{Age}(F^*)$ of $B$, $A^*$ embeds into $B^*$. Then $\text{Age}(F^*)$ has the expansion property relative to $\text{Age}(F)$.
\end{proposition}
\begin{proof}
Let $A \in \text{Age}(F)$, and enumerate the expansions of $A$ in $\text{Age}(F^*)$, which there are finitely many of by precompactness, as $(A^*_i)_{i=1}^n$. For each $i \in [n]$, let $B_i \in \text{Age}(F)$ be the structure provided by hypothesis for $A^*_i$. By the joint embedding property, let $B \in \text{Age}(F)$ embed every $B_i$. Then $B$ witnesses the expansion property for $A$.
\end{proof}

\begin{definition}
Let $\Lambda$ be a finite distributive lattice. Enumerate the meet-irreducibles as $E_i$, and for each $i$, let $E_i^+$ cover $E_i$. Then $\vec \AA_\Lambda^{\min}$ is the class of all expansions of elements of $\AA_\Lambda$ by a single subquotient order for each $i$, with bottom relation $E_i$ and top relation $E_i^+$.
\end{definition}

We first reduce the desired expansion property to one which will be easier to prove.

\begin{definition}
 Let $\KK'^{\min}$ be the closure under $\UU_\KK$-closed substructure of structures of the form $\Lift(\vec A)$ for some $\vec A \in \vec \AA_\Lambda^{\min}$. 
 
 Let $\KK'^{\min}_r$ be the reduct of $\KK'^{\min}$ forgetting the order.
 \end{definition}

\begin{lemma} \label{lemma:expansiontransfer}
Suppose $\KK'^{min}$ has the expansion property relative to $\KK'^{min}_r$. Then $\vec \AA_\Lambda^{\min}$ has the expansion property relative to $\AA_\Lambda$.
\end{lemma}
\begin{proof}
%Let $\vec A \in \vec \AA_\Lambda^{\min}$. By assumption, there is a $B' \in \KK'^{min}_r$ witnessing the expansion property for $\Lift(\vec A)$. By possibly enlarging $B'$, we may assume it is a reduct of $\Lift(\vec B)$ for some $\vec B \in \vec \AA_\Lambda^{min}$. We claim the reduct $B \in \AA_\Lambda$ of $\vec B$ witnesses the expansion property for $\vec A$.

%Let $\vec B_1 \in \vec \AA_\Lambda^{\min}$ be an arbitrary expansion of $B$.  Then $\Lift(\vec B_1) \in \KK'^{min}$ is a $\KK'^{min}$-expansion of $B$, and so $\Lift(\vec A)$ embeds into $\Lift(\vec B_1)$. But then we have $\vec A$ embeds into $\vec B_1$.

%Let $\vec A \in \vec \AA_\Lambda^{\min}$. By assumption, there is a $B \in \KK'_r$ witnessing the expansion property for $\Lift(\vec A)$. Let $B_0$ be the metric part of $B$, and consider $A_{B_0} \in \AA_\Lambda$. We claim $A_{B_0}$ witnesses the expansion property for $\vec A$.

%Let $\vec A_{B_0} \in \vec \AA_\Lambda^{\min}$ be an arbitrary expansion of $A_{B_0}$. Since $B_0 \subset A_{B_0}$, $\Lift(\vec A_{B_0}) \in \KK'^{min}$ contains a $\KK'^{min}$-expansion of $B$, and so $\Lift(\vec A)$ embeds into $\Lift(\vec A_{B_0})$. But then we have $\vec A$ embeds into $\vec A_{B_0}$.

Let $\vec A \in \vec \AA_\Lambda^{\min}$. By assumption, there is a $B \in \KK'^{min}_r$ witnessing the expansion property for $\Lift(\vec A)$. Let $B_0$ be the metric part of $B$, and consider $A_{B_0} \in \AA_\Lambda$. We claim $A_{B_0}$ witnesses the expansion property for $\vec A$.

Let $\vec A_{B_0} \in \vec \AA_\Lambda^{\min}$ be an arbitrary expansion of $A_{B_0}$. This induces an expansion $\vec B \in \KK'^{min}$ of $B$, such that $\vec A_{\vec B}$, as defined in Proposition \ref{proposition:representation}, is isomorphic to $\vec A_{B_0}$. 

Then $\vec B$ embeds $\Lift(\vec A)$, and therefore embeds $L_\KK(\vec A)$. Composing the natural injection of $\vec A$ into $L_\KK(\vec A)$ with the embedding of $L_\KK(\vec A)$ into $\vec B$, and composing the result with the map from $\vec B$ to $\vec A_{B_0}$ sending $x$ to $x_A$, gives an embedding of $\vec A$ into $\vec A_{B_0}$.
\end{proof}

We now use a standard argument (see \cite{digraphs}*{Theorem 8.6} for example, although it appeared earlier) to prove the expansion property for the lifted classes

\begin{lemma} \label{lemma:expansion}
$\KK'^{min}$ has the expansion property relative to $\KK'_r$.
\end{lemma}

\begin{proof}
Let $\vec A \in \KK'^{\min}$. Enumerate the meet-irreducibles in $\Lambda$ as $(E_i)_{i=1}^n$, and let $E_i^+$ cover $E_i$. For each $i$, consider the structure on $\set{x_1, x_2, x_3}$, such that
\begin{enumerate}
 \item $P_{E_i, 1}(x_1)$, $P_{E_i, 1}(x_2)$, $P_{E_i^+, 1}(x_3)$ \item $x_1, x_2 \leq_U x_3$
 \item $x_1 < x_2$
 \item The remaining order information is determined by $<_{1-types}$.
 \end{enumerate}
  Let $\vec p_i \in \KK'^{\min}$ be the minimal $\UU_\KK$-closed substructure of the \fraisse limit of $\KK'^{min}$ containing the above structure.
 
 By possibly enlarging $\vec A$, we may assume it has the form $\Lift(\vec B)$ for some $\vec B \in \vec \AA_\Lambda^{min}$. For each $i \in [n]$, let $<_{E_i}$ be the unique subquotient order on $\vec B$ from $E_i$ to $E_i^+$. For each $I \subset [n]$, let $\vec B_I$ be $\vec B$, but with $<_{E_i}$ reversed for every $i \in I$, and let $\vec A_I \cong \Lift(\vec B_I)$.

 Let $\vec C_0 \in \KK'^{min}$ embed $\vec A_I$ for every $I \subset [n]$, and let $\vec C_{i+1} \rightarrow (\vec C_i)^{\vec p_{i+1}}_2$ for every $i \in [n]$. Let $C_n \in \KK'_r$ be the reduct of $\vec C_n$. We will show $C_n$ witnesses the expansion property for $\vec A$.

Let $(C_n, \prec) \in \KK'^{min}$ be an expansion of $C_n$. For each $i \in [n]$, let $\prec_i$ be the restriction of $\prec$ to pairs $x, y \in P_{E_i, 1}$ such that $\delta(x, y) = E_i^+$, let $<_{E_i, 1}$ be the corresponding restriction of $<$, and let $\chi_i$ be a coloring of $\binom {\vec C_n}{\vec p_i}$ defined by
$$\chi_i(\vec p_i) =  \begin{cases} 
      0 & x_2 \prec_i x_1 \\
      1 & x_1 \prec_i x_2 \\
   \end{cases}$$
Then iterated applications of the Ramsey property yield $\vec C_0^* \subset \vec C_n$, a copy of $\vec C_0$ in which, for every $i \in [n]$, either $\prec_i = <_{E_i, 1}$, or $\prec_i = <_{E_i, 1}^{opp}$. Let $I \subset [n]$ be such that $\prec_i = <_{E_i, 1}^{opp}$ iff $i \in I$. We have that $\vec C_0^*$ contains a copy $\vec A_I^*$ of $\vec A_I$. Letting ${A_I}^*$ be the reduct of $\vec A_I^*$ to $\KK'_r$, we have $\vec A \cong \left(A_I^*, \set{\prec_i}_{i = 1}^n \right)$.
\end{proof}

\begin{theorem} \label{theorem:UMF}
 Let $\Lambda$ be a finite distributive lattice, $\AA_\Lambda$ the class of all finite $\Lambda$-ultrametric spaces, $\Gamma$  the \fraisse limit of $\AA_\Lambda$, and $\vec \Gamma^{min} = {(\Gamma, (<_{E_i})_{i=1}^n)}$ the \fraisse limit of $\vec \AA_\Lambda^{min}$. Then 
 \begin{enumerate}
 \item $\vec \AA_\Lambda^{\min}$ is a Ramsey class and has the expansion property relative to $\AA_\Lambda$.
 \item The logic action of $\text{Aut}(\Gamma)$ on $\overline{\text{Aut}(\Gamma) \cdot (<_{E_i})_{i=1}^n}$ is the universal minimal flow of $\text{Aut}(\Gamma)$.
 \end{enumerate}
\end{theorem}
\begin{proof}
By Theorem \ref{theorem:mainthm}, $\vec \AA_\Lambda^{\min}$ is a Ramsey class. By Lemmas \ref{lemma:expansiontransfer} and \ref{lemma:expansion}, $\vec \AA_\Lambda^{\min}$ has the expansion property relative to $\AA_\Lambda$. The second part then follows by Theorem \ref{theorem:NVT}.
\end{proof}

\end{document}